\documentclass[a4paper]{amsart}
\usepackage{amsmath,amsthm,amssymb}
\begin{document}
\newtheorem{thm}{Theorem}[section]
\newtheorem{lemma}[thm]{Lemma}
\newtheorem{proposition}[thm]{Proposition}
\newtheorem{corollary}[thm]{Corollary}
\newtheorem{conjecture}[thm]{Conjecture}

\theoremstyle{definition}
\newtheorem{defn}[thm]{Definition}

\newtheorem{remark}[thm]{Remark}
\newtheorem{remarks}[thm]{Remarks}
\newtheorem*{notation}{Main notation}

\newcommand{\arctg}{\operatorname{arctg}}
\newcommand{\sech}{\operatorname{sech}}
\newcommand{\ve}{\varepsilon}

\newcommand{\todoin}{\todo[inline]}

\numberwithin{equation}{section}

\title[Measure of instability of twist maps]{A new measure of instability and topological entropy of area-preserving twist diffeomorphisms}

\author{Sini\v{s}a Slijep\v{c}evi\'{c}, Zagreb}

\date{March 3rd, 2017}

\begin{abstract}
We introduce a new measure of instability of area-preserving twist diffeomorphisms, which generalizes the notions of angle of splitting of separatrices, and flux through a gap of a Cantori. As an example of application, we establish a sharp $>0$ lower bound on the topological entropy in a neighbourhood of a hyperbolic, unique action-minimizing fixed point, assuming only no topological obstruction to diffusion, i.e. no homotopically non-trivial invariant circle consisting of orbits with the rotation number 0. The proof is based on a new method of precise construction of positive entropy invariant measures, applicable to more general Lagrangian systems, also in higher degrees of freedom. 
 
\end{abstract}
 
\maketitle
%\keywords
%   To be completed
%\endkeywords

%\subjclass
%To be completed
%\endsubjclass

\section{Introduction}

The study of instabilities of Hamiltonian dynamical systems is one of the central themes of the theory of dynamical systems (see \cite{Bernard:10,Bernard:11,Cheng:10,Gelfreich:14,Kaloshin:15,Llave:06} for an overview and further references). In \cite{Slijepcevic:16}, we proposed a novel approach to Arnold's diffusion and construction of positive entropy invariant measures, by using recently developed techniques for study of dissipative partial differential equations \cite{Gallay:01,Gallay:12,Gallay:14,Gallay:15,Slijepcevic:13}. The goal of this paper is to apply these ideas to the simplest non-trivial example: a neighbourhood of a hyperbolic fixed point of an area-preserving twist diffeomorphism of the cylinder. We in particular introduce a new measure of instability, and show how it relates to two established notions: the flux through a gap in a Cantori \cite{MacKay:88}, and the angle of splitting of separatrices (\cite{Gelfreich:99} and references therein).

Recall \cite{Katok:95,MacKay:93} that we can define an area-preserving twist diffeomorphism as a map $f: (x,p) \mapsto (x',p')$, $f: \mathbb{S}^1 \times \mathbb{R} \rightarrow \mathbb{S}^1 \times \mathbb{R}$, and $(x',p')$ is uniquely defined by the following implicit equations:
\begin{eqnarray}
p &=& -V_1(x,x'), \nonumber \\
p' &=& V_2(x,x'). \nonumber
\end{eqnarray}
Here indices denote partial derivatives, and $V:\mathbb{R}^2 \rightarrow \mathbb{R}$ is the {\it generating function} satisfying for all $(x,x') \in \mathbb{R}^2$
\begin{itemize}
	\item[(A1)] $V$ is $C^2$, $V(x+1,x'+1)=V(x,x')$.
	
	\vspace{1ex}
	
	\item[(A2)] The twist condition: there exists $\delta >0$ so that $V_{12}(x,x')\leq -\delta$.
	
	\vspace{1ex}
	
	\item[(A3)] The second derivatives $V_{11}(x,x')$, $V_{12}(x,x')$, $V_{22}(x,x')$ are uniformly bounded and uniformly continuous on $(x,x')\in \mathbb{R}^2$.
\end{itemize}
(The condition (A3) appears to be somewhat more restrictive than the usual condition of convergence at infinity of $\lim_{|y|\rightarrow \infty}V(x,x+y)=\infty$, uniformly in $x$ \cite{Katok:95,MacKay:93}; however it is chosen purely for convenience and not an essential restriction, see Section 9.1.)
A typical example is the Standard (or Chirikov-Taylor) map given with $V(x,x')=(x'-x)^2/2 - k\cos(2 \pi x)/2\pi$, thus
\begin{eqnarray}
x' & = & x+p, \nonumber \\
p' & = & x+p + k\sin(2\pi x). \nonumber
\end{eqnarray}

We will require the following assumptions:
\begin{itemize}
	\item[(N1)] $y_0$ is the unique (mod 1) minimum of $x \mapsto V(x,x)$, and $(y_0,p_0)$, $p_0=-V_1(y_0,y_0)$ is a hyperbolic equilibrium of $f$.
	
	\vspace{1ex}
	\item[(N2)] There exists no homotopically non-trivial invariant circle consisting of orbits with the rotation number $0$. 
	
\end{itemize}
The definition of the rotation number is recalled in Section 2. The assumption (N1) is actually not needed, and chosen for the simplicity of the presentation (Section 9.2). The assumption (N2) is necessary, as shown by the time-one map of the mathematical pendulum, for which the results below do not hold. (N2) can be interpreted as the condition of no topological obstruction to diffusion and transport. (N1) and (N2) hold for the Standard map whenever $k \neq 0$ (Section 9.3).

Prior to stating the main results, we introduce the new notion of instability. Let $(z_0,p)$ be minimizing, homoclinic and $(\tilde{z}_0,\tilde{p})$ be the minimax homoclinic of $(y_0,p_0)$  (see Section \ref{s:preliminaries} for a precise definition of this and the forthcoming terms). Denote by $F$ a lift of $f$ on $\mathbb{R}^2$, by $\pi_1 : \mathbb{R}^2 \rightarrow \mathbb{R}$ the projection to the first coordinate, and let $\boldsymbol{z}=(z_k)_{k \in \mathbb{Z}}$, $\tilde{\boldsymbol{z}}=(\tilde{z}_k)_{k}$, $z_k=\pi_1(F^k(z_0,p))$, $\tilde{z}_k=\pi_1(F^k(\tilde{z}_0,\tilde{p}))$ be the {\it configurations} in $\mathbb{R}^{\mathbb{Z}}$ corresponding to the first coordinates of the orbits of $(z,p)$, $(\tilde{z},\tilde{p})$. We can always without loss of generality choose $F$, $(z_0,p)$ and $(\tilde{z}_0,\tilde{p})$ (see Lemma \ref{l:ordered} and the following paragraph) so that 
\begin{equation}
0 < \tilde{z}_0 < z_0 < \tilde{z}_1 < 1.
\end{equation}
Let $\Delta_0$ be the {\it Peierls barrier} (see Section 9.4), defined as
\begin{align*}
\Delta_0  =  \sum_{k=-\infty}^{\infty}(V(\tilde{z}_k,\tilde{z}_{k+1})-V(z_k,z_{k+1})).
\end{align*}
We recall in Section \ref{s:preliminaries} that (N1), (N2) imply $\Delta_0 > 0$, while always $\Delta_0 < \infty$ and the sum is absolutely convergent. We note that $\Delta_0$ is an analogue of the flux through a gap in a Cantori, introduced in \cite{MacKay:88} as a measure of instability and transport (see Section 9.5). We now introduce the new measures of instability:
\begin{equation}
\Delta_1 = \sup_{0 \leq e \leq \Delta_0/2} \inf_{\boldsymbol{h} \in \mathcal{N}(e)} ||\nabla E(\boldsymbol{h})||_{l^2(\mathbb{Z})}^2, \hspace{5ex} 
\tilde{\Delta}_1 = \sup_{0 \leq e \leq \Delta_0/2} \inf_{\boldsymbol{h} \in \tilde{\mathcal{N}}(e)} ||\nabla E(\boldsymbol{h})||_{l^2(\mathbb{Z})}^2, \label{d:Delta1}
 \end{equation}
where $\boldsymbol{h} \in l^2(\mathbb{Z})$, 
\begin{align}
 E(\boldsymbol{h}) & =\sum_{k=-\infty}^{\infty} (V(z_k+h_k,z_{k+1}+h_{k+1})-V(z_k,z_{k+1})), \label{d:e} \\
 \nabla E(\boldsymbol{h})_k & = V_2(z_{k-1}+h_{k-1},z_k+h_k)+V_1(z_k+h_k,z_{k+1}+h_{k+1}), \label{d:nablae}
\end{align}
and $\mathcal{N}(e)\subseteq \tilde{\mathcal{N}}(e)$ are suitable subsets of $\lbrace \boldsymbol{h} \in l^2(\mathbb{Z}), E(\boldsymbol{h})=e \rbrace$ defined in Section \ref{s:delta1}.

For example, in the case of the Standard map, $\Delta_1$ can be estimated for small $k$ (Section 4), and calculated by elementary methods in the {\it anti-integrable limit}, i.e. for large $k$, in which case it is $\Delta_1=k^2 + O(k)$ (Section 9.6).

Importantly, we show in Section \ref{s:delta1} that $\Delta_1 > \tilde{\Delta}_1 >0$ whenever (N1) and (N2) hold (hyperbolicity of $(y_0,p_0)$ is actually not required), and in Section \ref{s:interpretation} that, whenever stable and unstable manifolds intersect transversally, then $\Delta_1$ is proportional to the angle of splitting of separatrices. We will see that $\Delta_1$ quantifies instability, as stated in Theorems below.

Let $\lambda > 1$ be an eigenvalue of $Df(y_0,p_0)$ (there is always one by (N1); the other eigenvalue is then $1/\lambda$ because of the area-preserving property of $f$). Then by (N1), there exists a constant $ \kappa_1 > 0$ such that for every $j \geq 0$,
\begin{equation}
 |\tilde{z}_{-j}-y_0| \leq \kappa_1 \lambda^{-(j+1)}, \hspace{5ex} |\tilde{z}_j-y_0-1| \leq \kappa_1 \lambda^{-(j+1)}. \label{r:condhyp}
\end{equation}
Then we have:
\begin{thm} \label{t:main1}
	Assume (N1) and (N2) hold. Then there exists an ergodic, positive entropy invariant measure $\mu$ with metric entropy $h_{\mu}(f) \geq \log 2 / (2N)$, where $N$ is the smallest positive integer satisfying
	\begin{equation}
	N \geq  \frac{\log (4 \kappa_1 \kappa_2) - \log \Delta_1}{\log \lambda}. \label{r:Ncond}
	\end{equation} 
\end{thm}
Here $\kappa_2$ is a constant depending only on the first and second derivatives of $V$, explicitly given by (\ref{r:balance2}).
From the variational principle for topological and metric entropy \cite[Theorem 4.5.3]{Katok:95}, we deduce:

\begin{corollary}
	If (N1) and (N2) hold, the topological entropy $h_{\text{top}}(f) \geq \log 2 / (2N)$, where $N$ is as in Theorem \ref{t:main1}.
\end{corollary}

We note that the best lower bound above on the topological entropy is $\log 2/2$, thus not a good estimate for maps far from integrable, e.g. for the Standard map with large $k$. The reason for this is that our proof is optimized for maps close to integrable, and we embed the $2$-symbol shift in an invariant set of $f$. One could adjust the argument in the case of large $\Delta_1$ by embedding a shift with more symbols. For example, for large $k$ in the case of the Standard map, we can embed a $O(k)$-symbol shift, and so recover the lower bound $\log k + O(1)$ on the topological entropy as in \cite{Knill:96}.

Finally, we can approximate $(y_0,p_0)$ arbitrarily well by positive entropy measures, with known a-priori bounds from below on the entropy of measures in the sequence as a function of distance from $(y_0,p_0)$:

\begin{thm} \label{t:main2}
	There exists a sequence of ergodic positive entropy invariant measures $\mu_n$, weak$^*$-converging to $\delta_{(y_0,p_0)}$. Furthermore,
\begin{equation}
h_{\mu_n}(f)\geq \dfrac{\kappa_3}{N}  W_1\left( \mu_n,\delta_{(y_0,p_0)} \right) . \label{r:wasserstein}
\end{equation}
\end{thm}
Here $\delta_{(y_0,p_0)}$ is the Dirac measure concentrated at $(y_0,p_0)$, $W_1$ is the $L^1$-Wasserstein distance of measures, $N$ is as in Theorem \ref{t:main1}, and $\kappa_3$ is a constant depending on $\kappa_1$ and $\lambda$.

All these results can be generalized to action-minimizing periodic orbits satisfying analogues of (N1), (N2), and to irrational rotation numbers as long as the associated Mather set is a Cantori (Section 9.7), by associating a local measure of instability $\Delta_1(\rho)$ to each rotation number $\rho \in \mathbb{R}$. Furthermore, following the approach from \cite{Slijepcevic:16}, one can likely construct positive entropy 'diffusion' invariant measures, and embed Markov chains in a region of instability of an area-preserving twist diffeomorphism as outlined in \cite{MacKay:88}, with explicit estimates of the transport in terms of $\Delta_1(\rho)$.

Let us explain the main idea behind the construction and relate it to other approaches to estimating topological entropy. It was shown by Angenent \cite{Angenent:90,Angenent:92} and Boyland and Hall \cite{Boyland:87} that non-existence of a homotopically non-trivial invariant circle for a single rotation number $\rho \in \mathbb{R}$ implies $h_{top}(f)>0$ (our condition (N2) assumes that for $\rho=0$). In our understanding their proofs do not lead to lower bounds to topological entropy, as the methods are essentially topological. The only rigorous lower bounds on topological entropy of area-preserving twist maps we are aware of use some local uniformly hyperbolic structure and a shadowing argument to embed a shift map. Examples are the Standard map and small $k$, by using the estimates of the angle of intersection of stable and unstable manifolds of the fixed point by Lazutkin, Gelfreich and others \cite{Gelfreich:99}, or for large $k$ as done by Knill \cite{Knill:96}.

In \cite{Slijepcevic:16}, we proposed a method of constructing positive entropy measures of Lagrangian maps and flows, by constructing invariant sets of formally gradient dynamics of the action, without assuming any uniform hyperbolicity. We introduce the semiflow $\xi$ on a subset of $\mathcal{X}\subset \mathbb{R}^{\mathbb{Z}}$ associated to formally gradient dynamics of the action associated to an area-preserving twist map in (\ref{r:flow}). It is well-known that invariant sets of $\xi$ contain {\it configurations} $\boldsymbol{x}\in \mathcal{X}$ which are equilibria of $\xi$, thus correspond to true orbits of $f$ \cite{Gole:01,Slijepcevic:00}. The Angenent's proof of positive topological entropy \cite{Angenent:90,Angenent:92} can be interpreted in that context: Angenent associates to any sequence $\boldsymbol{\omega}\in \lbrace 0,1 \rbrace^{\mathbb{Z}}$ a sub-solution $\underline{\boldsymbol{x}}_{\boldsymbol{\omega}}$ and a super-solution $\overline{\boldsymbol{x}}_{\boldsymbol{\omega}}$ of $\xi$, such that $\underline{\boldsymbol{x}}_{\boldsymbol{\omega}} \leq \overline{\boldsymbol{x}}_{\boldsymbol{\omega}}$ (see Section \ref{s:preliminaries} for the partial order on $\mathcal{X}$). By the monotonicity of the semiflow $\xi$ (see Section \ref{s:construction}), the set $\underline{\boldsymbol{x}}_{\boldsymbol{\omega}} \leq \boldsymbol{x} \leq \overline{\boldsymbol{x}}_{\boldsymbol{\omega}}$ is $\xi$-invariant, thus contains an equilibrium $\boldsymbol{x}_{\boldsymbol{\omega}}$, which leads to establishing conjugacy of a 2-shift and $f^N|_{\mathcal{A}}$ for $N$ large enough, where $\mathcal{A}$ is an invariant set of $f^N$.

Angenent's approach does not extend to higher degrees of freedom, as it relies almost entirely on the ordering structure derived from the first coordinate being $x \in \mathbb{R}$. We use instead energy estimates developed recently to establish local control of dissipative PDEs, such as the Navier-Stokes equation \cite{Gallay:01,Gallay:12,Gallay:14,Gallay:15}, specifically the energy-energy dissipation-energy flux balance law stated here in an appropriate form in Lemma \ref{l:balance}. We refer the reader to \cite{Slijepcevic:16} for comparison of the method to other approaches to study of diffusion and transport in Hamiltonian systems, including the classical Arnold example \cite{Arnold:64,Bessi:96,Simo:00}.

The paper is structured as follows. We first recall the key facts of the Aubry-Mather theory and properties of the Peierls barrier $\Delta_0$. In Sections \ref{s:delta1} and \ref{s:interpretation} we establish positivity of $\Delta_1$ whenever (N1) and (N2) hold, and relate it to other notions of instability. We then recall and adapt the abstract tools of constructing positive entropy invariant measures from \cite{Slijepcevic:16}. The Section \ref{s:invariant} contains the core of the argument: a precise construction of a $\xi$-invariant set associated to any $\boldsymbol{\omega}\in \lbrace 0,1 \rbrace^{\mathbb{Z}}$. We complete the proofs in the next two sections, with some technical results moved to Appendices A and B.

\section{Preliminaries} \label{s:preliminaries}

In this section we recall the key facts of the Aubry-Mather theory \cite{Aubry:88,Mather:82,Sorrentino:10} and its extensions mainly due to Bangert and Gol\'{e} \cite{Bangert:94,Gole:92,Gole:01} needed in the following. We first introduce the notation and the key notions, and then establish the required properties of the minimising rays and $\Delta_0$, and in particular establish that (N1) and (N2) imply $\Delta_0>0$ (the condition (N1) is actually not required, as explained in Section 9.2).
 
Denote by $\mathcal{X}=\mathbb{R}^{\mathbb{Z}}$ equipped with the product topology. Let $\mathcal{E}$ be the set of all $\boldsymbol{x}$ such that for all $k \in \mathbb{Z}$, $V_2(x_{k-1}+h_{k-1},x_k+h_k)+V_1(x_k+h_k,x_{k+1}+h_{k+1})=0$, and then by the twist condition (A2),
\begin{equation}
\iota : \boldsymbol{x} \mapsto (x_0,p_0), \hspace{5ex} p_0=-V_1(x_0,x_1) \label{d:iota}
\end{equation} 
is a homeomorphism of $\mathcal{E}$ and $\mathbb{R}^2$, with $\boldsymbol{x}$ being the sequence of first coordinates of the $F$-orbit of $(x_0,p_0)$. Recall that $\boldsymbol{x}\in \mathcal{X}$ is action minimizing, if for any finite subsegment $n < m$, $V(x_n,x_{n+1},...,x_m) \leq V(x_n,\tilde{x}_{n+1},...\tilde{x}_{m-1},x_m)$ for any $(\tilde{x}_{n+1},...\tilde{x}_{m-1}) \in \mathbb{R}^{m-n-1}$, where $$V(x_n,x_{n+1},...,x_m)=\sum_{k=n}^{m-1} V(x_k,x_{k+1}).$$ The same definition naturally extends to segments $(x_n,x_{n+1},...,x_m)$ and rays $(x_k)_{k \leq k_0}$ or $(x_k)_{k \geq k_0}$.

Assume (N1) throughout the Section. Following Bangert \cite{Bangert:94}, we can study for any $x \in (y_0,y_0+1)$ the function
\begin{equation}
S(\boldsymbol{x})= \inf  \sum_{k=-\infty}^{\infty} (V(x_k,x_{k+1})-V(y_0,y_0)), \label{d:S}
\end{equation}
where infimum is taken over all $\boldsymbol{x} \in \mathbb{R}^{\mathbb{Z}}$ such that 
\begin{equation}
\lim_{k \rightarrow -\infty} x_k = y_0, \hspace{5ex} \lim_{k \rightarrow \infty} x_k = y_0+1, \label{r:Scond}
\end{equation} and such that $x_0=x$. The following is standard \cite{Bangert:94,Gole:01}:

\begin{lemma} For each $x \in (x_0,x_0+1)$, the infimum in (\ref{d:S}) is attained for some $\boldsymbol{x}\in \mathcal{X}$ such that $(x_k)_{k \leq 0}$ and $(x_k)_{k \geq 0}$ are minimizing rays.
\end{lemma}
Denote by $\mathcal{R} \subset \mathcal{X}$ the set of all $\boldsymbol{x}$ minimizing (\ref{d:S}) and satisfying (\ref{r:Scond}) for some $x_0 \in (y_0,y_0+1)$. We now have the following:
\begin{lemma} \label{l:rayproperties}
	(i) The function $ x \mapsto S(x)$ is continuous on $(y_0,y_0+1)$, bounded from below and above, and attains its minimum and maximum.
	
	(ii) For any $\boldsymbol{x} \in \mathcal{R}$ the following holds: $y_0 < x_k < y_0+1$, and $x_k$, $k \in \mathbb{Z}$ is a strictly increasing sequence.
	
	(ii) If $\boldsymbol{x} \in \mathcal{R}$ is such that $x_0$ is an extremal point of $x \mapsto S(x)$, then $\boldsymbol{x} \in \mathcal{E}$.
\end{lemma}
From now on we fix $\boldsymbol{z}, \tilde{\boldsymbol{z}}$ such that $z_0$, $\tilde{z}_0$ minimize, respectively maximize $S(x)$ on $(y_0,y_0+1)$. Now $\boldsymbol{z}$ is called minimizing, and $\tilde{\boldsymbol{z}}$ minimax homoclinic. We can now write $\Delta_0 = S(\tilde{x}_0)-S(x_0)$.

It is useful to introduce the partial order on $\mathcal{X}$ with $\boldsymbol{x} \leq \boldsymbol{y}$ and $\boldsymbol{x} \ll \boldsymbol{y}$ if for all $k \in \mathbb{Z}$, $x_k \leq y_k$, resp. $x_k < y_k$. We write $\boldsymbol{x} < \boldsymbol{y}$ if $\boldsymbol{x} \leq \boldsymbol{y}$ and $\boldsymbol{x} \neq \boldsymbol{y}$. We say that $\boldsymbol{x}$ and $\boldsymbol{y}$ are totally ordered if $\boldsymbol{x} \leq \boldsymbol{y}$  or $\boldsymbol{x} \geq \boldsymbol{y}$. Now the Aubry Fundamental Lemma \cite{Aubry:88} implies the following:
\begin{lemma} \label{l:ordered}
	If $z$, $\tilde{z}$ are a minimizing and a minimax homoclinic, then $z$ and $\tilde{z}$ are totally ordered.
\end{lemma}
From this and Lemma \ref{l:rayproperties} we deduce that without loss of generality, we can choose $\boldsymbol{z}$, $\tilde{\boldsymbol{z}}$ so that
\begin{equation}
y_0< ... <z_{-1}<\tilde{z}_0<z_0<\tilde{z}_1 < z_1 < \tilde{z}_2 < .... < y_0+1
\end{equation}

Denote by $\boldsymbol{y}_0$ the constant configuration $y_0$, thus by (N1) $\boldsymbol{y}_0,\boldsymbol{y}_0+1 \in \mathcal{E}$. Let $\pi_0 : \mathcal{X} \rightarrow \mathbb{R}$ be the projection $\pi_0 : \boldsymbol{x} \mapsto x_0$. We can now deduce the following:
\begin{lemma} \label{l:delta0} If $\Delta_0 = 0$, then the projection $\pi_0 : \mathcal{R}\cup \lbrace \boldsymbol{y}_0,\boldsymbol{y}_0+1 \rbrace \rightarrow [y_0,y_0+1]$ is a homeomorphism.
\end{lemma}
\begin{proof}
By the assumption $\Delta_0=0$, $\pi_0$ is onto $(y_0,y_0+1)$ as for each $w \in (y_0,y_0+1)$ we can find a minimizing ray $\boldsymbol{w} \in \mathcal{R}$, $w_0=w$; thus $\pi_0$ is onto $[y_0,y_0+1]$. As the domain $\mathcal{R}\cup \lbrace \boldsymbol{y}_0,\boldsymbol{y}_0+1 \rbrace$ consists of minimizing configurations, by the Aubry Fundamental Lemma \cite{Aubry:88,Gole:01} any two configurations in $\mathcal{R}\cup \lbrace \boldsymbol{y}_0,\boldsymbol{y}_0+1 \rbrace$ can not intersect, thus $\pi_0$ is injective. By definition $\pi_0$ is continuous, thus it must be a homeomorphism.
\end{proof}

Let $\rho(\boldsymbol{x}) = \lim_{|k| \rightarrow \infty} x_k/|k|$ be the rotation number of $\boldsymbol{x} \in \mathcal{X}$ whenever defined, and let $\rho((x,p))=\rho(\iota^{-1}(x,p))$, $(x,p) \in \mathbb{R}$.

\begin{corollary} \label{c:25} If (N2) holds, then $\Delta_0 > 0$.
\end{corollary}

\begin{proof} Assume $\Delta_0=0$. 
Consider the image of $\mathcal{R}\cup \lbrace \boldsymbol{y}_0 \rbrace$ under the composition $\hat{\iota}$ of $\iota$ and the standard projection $\mathbb{R}^2 \mapsto \mathbb{S}^1 \times \mathbb{R}$. Lemma \ref{l:delta0} implies that it is a homotopically non-trivial invariant circle of $f$ consisting of orbits with the rotation number $0$, parametrized by $\hat{\iota} \circ \pi_0^{-1}$, which contradicts (N2).
\end{proof}

\section{Positivity of $\Delta_1$} \label{s:delta1} \label{s:positivity}

We now show:

\begin{proposition} \label{p:positive}
	If (N2) holds, then $\Delta_1 \geq \tilde{\Delta}_1 > 0$.
\end{proposition}

The main tool will be an infinite-dimensional version of the Morse-Sard lemma valid for operators on $l^2(\mathbb{Z})$, due to Poho\v{z}aev \cite{Pohozaev:68}. We first establish some basic properties of the functional $E : l^2(\mathbb{Z}) \rightarrow \mathbb{R}$ defined by (\ref{d:e}), and then define the sets $\mathcal{N}(e)$, $\tilde{\mathcal{N}}(e)$, thus completing the definition of $\Delta_1$, $\tilde{\Delta}_1$. Following that, we recall the Morse-Sard-Poho\v{z}aev lemma, and then complete the proof by a simple compactness argument.

\begin{lemma} \label{l:eprop}
	(i) $E: l^2(\mathbb{Z}) \rightarrow \mathbb{R}$ is well-defined and $C^2$.
	
	(ii) For all $\boldsymbol{h},\boldsymbol{x} \in l^2(\mathbb{Z})$, we have $\nabla E(\boldsymbol{h}) \in l^2(\mathbb{Z})$, and $DE(\boldsymbol{h})\boldsymbol{x}= \langle \nabla E(\boldsymbol{h}),\boldsymbol{x} \rangle $ (the scalar product on $l^2(\mathbb{Z})$).
	
	(iii) For all $\boldsymbol{h} \in l^2(\mathbb{Z})$, the Kernel of $D^2E(\boldsymbol{h})$ is at most 2-dimensional.
\end{lemma}

\begin{proof}
Define for any $\boldsymbol{h} \in l^2(\mathbb{Z})$ a tridiagonal $\mathbb{Z} \times \mathbb{Z}$ matrix $A(\boldsymbol{h})$ given with $A=[a_{i,j}]_{i,j \in \mathbb{Z}}$,
\begin{equation*}
a_{i,j}= \begin{cases} V_{12}(z_{i-1}+h_{i-1},z_i+h_i) & j=i-1, \\
 V_{22}(z_{i-1}+h_{i-1},z_i+h_i) +  V_{11}(z_i+h_i,z_{i+1}+h_{i+1}) & j=i, \\
 V_{12}(z_i+h_i,z_{i+1}+h_{i+1}) & j=i+1, \\
 0 & \text{otherwise}.
 \end{cases}
\end{equation*} 
The condition (A3) implies that $\boldsymbol{x} \rightarrow A(\boldsymbol{h})\boldsymbol{x}$ is a well-defined continuous linear operator on $l^2(\mathbb{Z})$, with the norm
$$||A(\boldsymbol{h})||_{l^2(\mathbb{Z})} \leq 3 \sup_{i \in \mathbb{Z}} (|V_{11}((z_i+h_i,z_{i+1}+h_{i+1})|+|V_{12}((z_i+h_i,z_{i+1}+h_{i+1})|+|V_{22}((z_i+h_i,z_{i+1}+h_{i+1})|).
$$
An analogous estimate of the norm of $A(\tilde{\boldsymbol{h}})-A(\boldsymbol{h})$ and uniform continuity in (A3) implies that $\boldsymbol{h} \mapsto A(\boldsymbol{h})$ is continuous. Now as $E(\boldsymbol{0})=0$ and $\nabla E(\boldsymbol{0})=0$ (as $\boldsymbol{z} \in \mathcal{E}$), it is easy to verify by that for any $\boldsymbol{x} \in l^2(\mathbb{Z})$,
\begin{equation*}
E(\boldsymbol{h})  = \frac{1}{2}\langle A(\boldsymbol{c})\boldsymbol{h},\boldsymbol{h} \rangle, \hspace{5ex}
\nabla E(\boldsymbol{h})\boldsymbol{x}  = A(\boldsymbol{c}(\boldsymbol{x}))\boldsymbol{x}
\end{equation*}
for some $\boldsymbol{c},\boldsymbol{c}(\boldsymbol{x}) \in [0,\boldsymbol{h}]$, thus $E(\boldsymbol{h})$ and 
$\nabla E(\boldsymbol{h})\boldsymbol{x}$ are well defined. By straightforward calculation it follows that $\nabla E(\boldsymbol{h})$ and $A(\boldsymbol{h})$ are representatives in $l^2(\mathbb{Z})$, resp. $L(l^2(\mathbb{Z}),l^2(\mathbb{Z}))$ of the first, respectively second differential of $E$ at $\boldsymbol{h}$, which gives the claims (i) and (ii).

We now note that $\text{Ker }A(\boldsymbol{h})$ corresponds to the set of tangential orbits at $\boldsymbol{z}+\boldsymbol{h}$, thus the set $\boldsymbol{x} \in \mathcal{X}$, $A(\boldsymbol{h})\boldsymbol{x} = 0$ is homeomorphic to $\mathbb{R}^2$ (equivalently, one can uniquely solve $A(\boldsymbol{h})\boldsymbol{x} =0$ for $ \boldsymbol{x} \in \mathcal{X}$ given any $x_0,x_1 \in \mathbb{R}$ by applying (A2) and the definition of $A(\boldsymbol{h})$), which completes the proof of (iii).
\end{proof}
Let $ u_j = 2\kappa_1 \lambda^{-|j|}, j \in \mathbb{Z}$, where $\kappa_1,\lambda$ are as in (\ref{r:condhyp}). Clearly $\boldsymbol{u}\in l^2(\mathbb{Z})$. Define $\tilde{\mathcal{N}}(e)$ as the set of all $\boldsymbol{h} \in l^2(\mathbb{Z}) \cap \mathcal{X}_2$ such that $E(\boldsymbol{h})=e$, and such that for all $j \in \mathbb{Z}$, $|h_j| \leq u_j$. Let $\mathcal{C}(e)$ be the connected component of $0$ of the set $\lbrace \boldsymbol{h} \in l^2(\mathbb{Z}), E(\boldsymbol{h}) \leq e \rbrace$, and let
$$ \mathcal{N}(e)= \tilde{\mathcal{N}}(e) \cap \mathcal{C}(e).$$
It is easy to verify that $\tilde{\mathcal{N}}(e)$ is a compact subset of $l^2(\mathbb{Z})$, and that $\mathcal{N}(e)$ is a closed subset of $\tilde{\mathcal{N}}(e)$, thus compact.

Recall the Morse-Sard-Poho\v{z}aev lemma \cite{Pohozaev:68} holding for Fredholm functionals on real, separable, reflexive Banach spaces, thus valid on $l^2(\mathbb{Z})$:
\begin{lemma} \label{l:pohozaev}
Assume that $\tilde{E}:l^2(\mathbb{Z}) \rightarrow \mathbb{R}$ is a real $C^k$ functional, that $\dim (\operatorname{Ker} D^2  {\tilde{E}}(\boldsymbol{h})) \leq m < \infty$ for any $\boldsymbol{h} \in l^2(\mathbb{Z})$, and that $k \geq \max \lbrace m,2 \rbrace$. Then the set of critical values of $\tilde{E}$ has Lebesgue measure 0.
\end{lemma}
Here the critical value of $\tilde{E}$ is any $e \in \mathbb{R}$ such that for some $\boldsymbol{h} \in l^2(\mathbb{Z})$, $D\tilde{E}(\boldsymbol{h})=0$. Clearly by Lemma \ref{l:eprop}, Lemma \ref{l:pohozaev} is applicable to $\tilde{E}=E$ with $k=m=2$.

First note that $\Delta_1 \geq \tilde{\Delta}_1$ follows from the definition (\ref{d:Delta1}) and $\mathcal{N}(e) \subseteq \tilde{\mathcal{N}}(e)$. By Lemma \ref{l:pohozaev}, we now find any non-critical value $0< e_0 < e$ of $E$. As now $||\nabla E(\boldsymbol{h})||^2_{l^2(\mathbb{Z})} > 0$ on $\tilde{\mathcal{N}}(e_0)$, $\boldsymbol{h} \mapsto ||\nabla E(\boldsymbol{h})||^2_{l^2(\mathbb{Z})}$ is continuous and $\tilde{\mathcal{N}}(e_0)$ is compact, we can bound $||\nabla E(\boldsymbol{h})||^2_{l^2(\mathbb{Z})}$ from below away from $0$ on $\tilde{\mathcal{N}}(e_0)$ which completes the proof of Proposition \ref{p:positive}.

\section{Interpretation of $\Delta_1$} \label{s:interpretation}

We now explicitly relate $\Delta_1$ to the angle of splitting of separatrices. We do it by recalling the notion of phonon gap, denoted by $\Delta_2$ below. We first relate $\Delta_2$ to the minimal angle of splitting of stable and unstable bundles along the orbit of the minimizing homoclinic by using the results of Aubry, MacKay and Baesens \cite{Aubry:92}, then show that $\Delta_1 \gg \Delta_2^4$, and finally state a conjecture on the relation of $\Delta_1$ and $\Delta_0$.

Let
\begin{equation}
 \Delta_2 = || D^2 E(\boldsymbol{0})^{-1} ||_{l^2(\mathbb{Z})}^{-1} \label{r:def2}
\end{equation}
We have shown in Lemma \ref{l:eprop},(i) that $E$ is $C^2$, thus $\Delta_2$ is well defined whenever $D^2 E(\boldsymbol{0})$ is invertible; we set $\Delta_2=0$ otherwise. The quantity $\Delta_2$ is exactly the phonon gap of the 'equilibrium state' corresponding to the minimizing homoclinic $(z,p)$, as defined in \cite{Aubry:92}.

\begin{remark} \label{r:41} Assume (N1) and that the stable and unstable manifold of $(y_0,p_0)$ intersect transversally at $(z,p)$. Then as shown in \cite[Proposition 2, (2.56)]{Aubry:92}, $\Delta_2$ can be explicitly bounded from below by
	$$
	\Delta_2 \geq \frac{\lambda - 1}{2 \kappa_1 \alpha \kappa_4},
	$$
where $\lambda, \kappa_1$ are as in (\ref{r:condhyp}), and $\alpha$, $\kappa_4$ are constants inverse proportional to the minimal angle of splitting of stable and unstable bundles along the orbit of $(z,p)$, and constants depending only on the second derivatives of $V$ (refer to \cite{Aubry:92} for a precise definition of $\alpha$, $\kappa_4$). Note that the angle of splitting of stable and unstable bundles is typically minimal exactly at $(z,p)$ \cite{Gelfreich:99}.
\end{remark} 

\begin{lemma} Assume (N1),(N2), and that $V$ is $C^3$. Then
	\begin{equation}
	 \Delta_1 \geq \kappa_5 \Delta_2^4,   
	\end{equation}
where $\kappa_5$ is an explicit constant depending on the third derivatives of $V$.
\end{lemma}

\begin{proof} The constants $\kappa_5, \kappa_6$ used below depend only on the bounds on third derivatives of $V$ on $|x-x'| \leq 3$. As $V$ is $C^3$, it is easy to verify as in the proof of Lemma \ref{p:positive} that $E$ is $C^3$ on the interior of $l^2(\mathbb{Z}) \cap \mathcal{X}_3$. We now see that for any $\boldsymbol{h} \in l^2(\mathbb{Z})\cap \mathcal{X}_2$ (and as $\boldsymbol{z} \in \mathcal{X}_1$) we have that 
\begin{align*}
E(\boldsymbol{h})& =\frac{1}{2}\langle D^2E(\boldsymbol{0})\boldsymbol{h},\boldsymbol{h} \rangle/2 + O_V\left( ||\boldsymbol{h}||_{l^2(\mathbb{Z})}^3\right) \geq \frac{1}{2} \Delta_2 ||\boldsymbol{h}||^2_{l^2(\mathbb{Z})} + O_V\left( ||\boldsymbol{h}||_{l^2(\mathbb{Z})}^3 \right), \\
||\nabla E(\boldsymbol{h}) ||_{l^2(\mathbb{Z})} & =\left| \left| D^2 E(\boldsymbol{0})\boldsymbol{h} \right| \right|_{l^2(\mathbb{Z})} + O_V(||\boldsymbol{h}||_{l^2(\mathbb{Z})}^2)  \geq \Delta_2 ||\boldsymbol{h}||_{l^2(\mathbb{Z})} +  O_V\left(||\boldsymbol{h}||_{l^2(\mathbb{Z})}^2 \right),
\end{align*}
thus 
\begin{align*}
\left|\frac{||\nabla E(\boldsymbol{h}) ||^2_{l^2(\mathbb{Z})} }{2 \Delta_2} - E(\boldsymbol{h}) \right| & \leq \kappa_6 ||\boldsymbol{h}||_{l^2(\mathbb{Z})}^3, \\
\left| E(\boldsymbol{h}) - \frac{\Delta_2}{2}||\boldsymbol{h}||_{l^2(\mathbb{Z})}^2\right|& \leq \kappa_6 ||\boldsymbol{h}||_{l^2(\mathbb{Z})}^3.
\end{align*}
Now it is easy to see that it suffices to choose $e_0 = \Delta_2^3 / 32 \kappa_6^2$ and recall the definition of $\mathcal{N}(e_0)$ in the previous section, to obtain that for each $\boldsymbol{h} \in \mathcal{N}(e_0)$ we have that $||\nabla E(\boldsymbol{h}) ||^2_{l^2(\mathbb{Z})} /(2 \Delta_2) \geq e_0 / 2$, which completes the proof with $\kappa_5 = 1 / (32 \kappa_6^2)$.
\end{proof}

\begin{remark} \label{r:43}
Assuming (N1), (N2), one can similarly establish $\Delta_0 \geq \kappa_7 \Delta_2^3$, where $\kappa_7$ is also an explicit constant depending only on the third derivatives of $V$. The proof is omitted.
\end{remark}

Finally, we conjecture the following:

\begin{conjecture}
For a generic $C^2$ area-preserving twist diffeomorphism, $\Delta_1 \geq \Delta_0^2$. 
\end{conjecture}
The rationale for this is that intuitively $\Delta_1^{1/2}$ is expected to be generically proportional to the maximal derivative of $s \mapsto E(\boldsymbol{h}(s))$ on $[0,1]$, where $s \mapsto \boldsymbol{h}(s)$ is the straight line connecting $\boldsymbol{h}(0)=\boldsymbol{0}$ and $\boldsymbol{h}(1)=\tilde{\boldsymbol{z}}-\boldsymbol{z}$, thus $E(\boldsymbol{h}(1))-E(\boldsymbol{h}(0))=\Delta_0$.

\section{Shadowing invariant measures}

We recall the definition of a shadowing invariant measure, introduced in \cite{Slijepcevic:16}. In this section we will always consider a measurable space $(\Omega,\mathcal{F})$, where $\Omega$ is a compact metric space and $\mathcal{F}$ the Borel $\sigma$-algebra. Let $S$ be a homeomorphism on $\Omega$ and $\mu$ a $S$-invariant probability measure on $(\Omega,\mathcal{F})$. 

\begin{defn} Let $\mathcal{G}$ be a $\sigma$-subalgebra of $\mathcal{F}$. We say that a $S$-invariant Borel-probability measure $\nu$ $\mathcal{G}$-shadows a $S$-invariant probability measure $\mu$ on $(\Omega,\mathcal{F})$, if $\mu$ is a factor of $\nu$, and if for each $\mathcal{D} \in \mathcal{G}$, we have $\mu(\mathcal{D})=\nu(\mathcal{D})$. 
\end{defn}

The notion of shadowing measure is an ergodic-theoretical analogue of the notion of a shadowing orbit, as it will be clear from its application later. The following relation of shadowing to ergodicity will be useful.

\begin{lemma} \label{l:muergodic}
	Assume that $\nu$ $\mathcal{G}$-shadows $\mu$, that $\mu$ is $S$-ergodic, and that $\mathcal{G}$ satisfies the following: for each $\mathcal{D} \in \mathcal{G}$, $\theta^{-1}(\mathcal{M}_1 \cap \mathcal{D}) \subset \mathcal{D}$. Then almost every measure in the ergodic decomposition of $\nu$ $\mathcal{G}$-shadows $\mu$.  
\end{lemma}

For convenience of the reader, we repeat the short proof from \cite{Slijepcevic:16}.

\begin{proof} Consider the ergodic decomposition of $\nu$, i.e. a Borel-probability measure $\chi$ on the compact, metrizable space of probability measures $\mathcal{M}(\Omega)$ (equipped with the weak$^*$-topology), such that $\chi$-a.e. measure is $S$-invariant and ergodic, and such that the usual representation formula for $\nu$ in terms of $\chi$ holds \cite{Walters}. Then it is straightforward to check by verifying the definition of the ergodic decomposition \cite{Walters} that $(\theta^*)^*\chi$ is the ergodic decomposition of $\mu$, where $(\theta^*)^*$ is the double push defined in a natural way. However, the ergodic decomposition is unique, and as $\mu$ is ergodic, $(\theta^*)^*\chi$ must be concentrated on $\mu$. That means that for $\chi$-a.e. measure $\tilde{\nu}$ (i.e. almost every measure in the ergodic decomposition of $\nu$), we have $\theta^*(\tilde{\nu})=\mu$. By construction, $\mu$ is then a factor of $\tilde{\nu}$.
	
	It remains to show the shadowing property. As $\mu(\mathcal{M}_1)=1$, we have
	\begin{equation}
	\tilde{\nu}(\mathcal{\mathcal{D}}) \geq \tilde{\nu}(\theta^{-1}(\mathcal{M}_1 \cap \mathcal{D})) = \mu ( \mathcal{M}_1 \cap \mathcal{D}) = \mu (\mathcal{D}), \label{r:factor}
	\end{equation}
	and analogously $\tilde{\nu}(\mathcal{D}^c)\geq \mu (\mathcal{D}^c)$. However, $1 = \mu (\mathcal{D})+\mu (\mathcal{D}^c)= \tilde{\nu}(\mathcal{\mathcal{D}}) +\tilde{\nu}(\mathcal{D}^c)=1$. We conclude that the equality in (\ref{r:factor}) must hold.
\end{proof}
As entropy is always non-increasing under factor maps \cite{Walters}, the following follows directly from the definition:

\begin{lemma} If $\nu$ $\mathcal{G}$-shadows $\mu$, then $h_{\nu}(S) \geq h_{\mu}(S)$.
	\label{l:entropy}
\end{lemma}

\section{Construction of shadowing invariant measures} \label{s:construction}

In this section we introduce the formally gradient semiflow $\xi$ of the action, and then use it to explicitly construct invariant measures of $f$ with a-priori specified ergodic properties. We do it by constructing invariant (Borel probability) measures of $\mathcal{X}$ with respect to the shift $S$, $(S\boldsymbol{x})_k=x_{k+1}$, as equilibria of the induced semiflow $\xi^*$ on the space of $S$-invariant measures on $\mathcal{X}$.

Let $\mathcal{X}_K$, $K \in \mathbb{N}$ to be the subset of all configurations $\boldsymbol{x} \in \mathcal{X}$, such that $\sup_{j \in \mathbb{Z}}|x_j-x_{j+1}| \leq K$. We consider the system of equations on $\mathcal{X}_K$
\begin{equation}
\frac{dx_j(t)}{dt}=-V_2(x_{j-1}(t),x_j(t))-V_1(x_j(t),x_{j+1}(t)). \label{r:flow}
\end{equation}
The flow can be understood as formally gradient, as it can be formally written as
$$ \frac{d\boldsymbol{x}(t)}{dt}=-\nabla \Phi (\boldsymbol{x}(t)),$$
where $\Phi(x)=\sum_{j \in \mathbb{Z}}V(x_j,x_{j+1})$ is a formal sum - the action of $x=(x_j)_{j\in \mathbb{Z}}$. Clearly $\Phi(x)$ is typically not well-defined; and the system of equations (\ref{r:flow}) does not typically generate a gradient flow. Intricacies of its dynamics, and the related theory of extended gradient systems, is discussed in some detail in \cite{Slijepcevic:13}. Let $\hat{\mathcal{X}}=\mathcal{X}/\mathbb{Z}$, where the quotient with respect to $\mathbb{Z}$ is obtained by identifying $\boldsymbol{x}+n$ for all $n\in \mathbb{Z}$, and let $\hat{\mathcal{X}}_K = \mathcal{X}_K /\mathbb{Z}$. Denote by $\hat{S}$ the induced map $S$ on $\hat{\mathcal{X}}$, $\hat{\mathcal{X}}_K$, and let $\hat{\mathcal{E}}=\mathcal{E}/\mathbb{Z}$ and $\hat{\iota}=\hat{\pi} \circ \iota$, $\hat{\pi} : \mathcal{X} \rightarrow \tilde{\mathcal{X}}$ the projection. It is well-known that (\ref{r:flow}) generates continuous semiflows $\xi$ and $\hat{\xi}$ on $\mathcal{X}_K$, resp.  $\hat{\mathcal{X}}_K$ for all $K \in \mathbb{N}$. (In particular, $\mathcal{X}_K$ and $\hat{\mathcal{X}}_K$ are invariant, see e.g. \cite{Slijepcevic:13}.)

Denote by $\mathcal{M}(\hat{\mathcal{A}}, \hat{S})$ the set of Borel probability $\hat{S}$-invariant measures on a Borel-measurable subset $\hat{\mathcal{A}}$ of $\hat{\mathcal{X}}$. Furthermore, the semiflow $\xi$ has the well-known order-preserving property (\cite{Gole:01} and references therein) analogous to the maximum principle for parabolic semilinear scalar differential equations, needed later in the following form for rays:

\begin{lemma} \label{l:order}
	Assume $\tilde{\boldsymbol{x}} \in \mathcal{E}$ and $\lambda >0$. Then:
	(i) If 
	\begin{equation}
	x_{-j}(t)  \leq \tilde{x}_{-j}, \: \: j \geq 0 \label{r:left}
	\end{equation}
	for $t=t_0$, and $x_0(t) \leq \tilde{x}_0$ for $t \in [t_0,t_0+\lambda]$, we have that (\ref{r:left}) holds for all $t \in [t_0,t_0+\lambda]$.
	
	\vspace{1ex}
	\noindent(ii) If
	\begin{equation} 
	x_j(t)  \geq \tilde{x}_j,  \: \: j \geq 1, \label{r:right}
	\end{equation}
	for $t=t_0$, and $x_1(t) \geq \tilde{x}_1$ for $t \in [t_0,t_0+\lambda]$, we have that (\ref{r:left}) holds for all $t \in [t_0,t_0+\lambda]$.
\end{lemma}

We first establish why it suffices to consider $\hat{S}$-invariant measures on $\hat{\mathcal{E}}$.

\begin{lemma} \label{l:push}
Let $\mu \in \mathcal{M}(\hat{\mathcal{E}}, \hat{S})$. Then the pushed measure $\hat{\iota}^* \mu$ is an $f$-invariant measure isomorphic to $\mu$.
\end{lemma}

\begin{proof}
	This follows from easily verifiable facts that $\hat{\iota} : \hat{\mathcal{E}} \rightarrow \mathbb{S}^1 \times \mathbb{R}$ is a homeomorphism, and that $\hat{\iota} \circ \hat{S} = f \circ \hat{\iota}$.
\end{proof}

Assume now $\mu \in \hat{\mathcal{X}}$, and let $\mathcal{G}$ be a $\sigma$-subalgebra of Borel sets of $\hat{\mathcal{X}}$, satisfying the following:

\begin{itemize}
	\item[(M1)] {\bf The separation property.} There exists a Borel-measurable set $\mathcal{M}_1 \subset \hat{\mathcal{X}}$ such that $\mu(\mathcal{M}_1)=1$, and such that $\left\lbrace \mathcal{D} \cap \mathcal{M}_1, \: \mathcal{D} \in \mathcal{G} \right\rbrace$ generates all Borel-measurable sets on $\mathcal{M}_1$. Specifically, for each $q \in \mathcal{M}_1$, there exists $\mathcal{D}_q \in \mathcal{G}$, $q \in \mathcal{D}_q $ such that 
	if $q,\tilde{q} \in \mathcal{M}_1$, $q \neq \tilde{q}$, then
	$\mathcal{D}_q \cap \mathcal{D}_{\tilde{q}} = \emptyset$. Furthermore, for any $q \in \mathcal{M}_1$, $\mathcal{D}_{\hat{S}(q)}=\hat{S}(\mathcal{D}_q)$.
	
	\item[(M2)] {\bf Measurability.} If $\mathcal{M}_2 = \cup_{q \in \mathcal{M}_1}{\mathcal{D}_q}$, then the map $\hat{\theta} : \mathcal{M}_2 \rightarrow \mathcal{M}_1$ given with $\hat{\theta}(\mathcal{D}_q)=q$ is Borel-measurable. Specifically, $\mathcal{M}_2$ is Borel-measurable.
	
	\item[(M3)] {\bf The closed-sets property.} There exists a family $\mathcal{D}_i \in \mathcal{G}$ of closed sets, $i \in \mathcal{I}$, such that $\mathcal{G}$ is generated by this family Furthermore, for each $i_1 \in \mathcal{I}$ there exists a sequence $i_n \in \mathcal{I}$, $n \in \mathbb{N}$ such that $\mathcal{D}_{i_n}$ are pairwise disjoint, and such that $\mu(\cup_{n=1}^{\infty}\mathcal{D}_{i_n})=1$. 
	
	\item[(M4)] {\bf The $\xi$-invariance.} For each $q \in \mathcal{M}_1$ and each $\mathcal{D} \in \mathcal{G}$, if $q \in \mathcal{D}$, then for all $t \geq 0$, $\hat{\xi}^t(q) \in \mathcal{D}$.
\end{itemize}

In applications, (M1)-(M3) will typically trivially follow from the construction, with the core of the argument being in demonstrating (M4).

\begin{proposition} \label{p:construction}
		Assume $\mu \in \mathcal{M}(\hat{\mathcal{X}}_K, \hat{S})$ for some $K \in \mathbb{N}$. Let $\mathcal{G}$ be a $\sigma$-subalgebra of Borel sets on $\hat{\mathcal{X}}$ satisfying (M1)-(M4). Then there exists $\nu \in \mathcal{M}(\hat{\mathcal{E}}, \hat{S})$ which $\mathcal{G}$-shadows $\mu$.
		
		Furthermore, if $\mu$ is $\hat{S}$-ergodic, we can choose $\nu$ to be $\hat{S}$-ergodic.
\end{proposition}

The proof of Proposition \ref{p:construction} is analogous to the proof of \cite[Theorem 4.4]{Slijepcevic:16}. For convenience of the reader, we give it in the Appendix A.

\begin{remark} \label{r:N}
	Let $N \geq 1 $ be an integer. Then the statements of Lemma \ref{l:push} and Proposition \ref{p:construction} hold if we replace $\hat{S}$, $f$ with $\hat{S}^N$, $f^N$ everywhere, including the shift-invariance in (M1); the proofs are analogous. 
\end{remark}

\section{Invariant sets of the formally gradient semiflow} \label{s:invariant}

Let $\Omega=\lbrace 0,1 \rbrace^{\mathbb{Z}}$ equipped with the product topology. The main result of this section is a construction of a non-empty, closed, $\xi$-invariant subset of $\mathcal{X}_1$ associated to any $\boldsymbol{\omega}=(\omega_i)_{i \in \mathbb{Z}} \in \Omega$. 

We first define the notation, and then state the result. The 'nesting' two-step approach to proving $\xi$-invariance is specified in Lemma \ref{l:nesting}. The two steps are based on an order-preserving, and an energy argument.

Fix $\boldsymbol{\omega} \in \Omega$. Let $j_{k}, k \in \mathbb{Z}$ be the increasing sequence of positions of $1$ in $\boldsymbol{\omega}$ such that $j_0 \leq 0 < j_1$ (the construction can be adapted in a straightforward way to the case of finitely many or no ones in $\boldsymbol{\omega}$, thus not discussed separately). Let $N>0$ be an integer constant specified later (we omit dependency of $N$ in the notation introduced below for simplicity).

The idea is that we 'glue' $2N$-large pieces of the minimizing equilibrium $\boldsymbol{y}_0 \equiv y_0$ whenever $\omega_i = 0$, and of the minimizing homoclinic $\boldsymbol{z}$ whenever $\omega_i = 1$. Specifically, let $m \in \mathbb{Z}$, $-N+1 \leq j ... \leq N$, and $k$ such that $j_k \leq m < j_{k+1}$ (e.g. for positive $m$, $k$ is the number of 'ones' between 0 and $m$). Then we define
\begin{align*}
 x^{\boldsymbol{\omega}}_{2mN+j} = \begin{cases}
   y_0 + k & \omega_m = 0,  \\
   z_j + k & \omega_m = 1.
 \end{cases}
\end{align*}
Furthermore, we combine minimax homoclinics, to be useful shortly:
\begin{align*}
 x^{\boldsymbol{\omega},-}_i&=\tilde{z}_j+k & i&=2i_kN+j, \: k\in \mathbb{Z}, \: j=1,...,2(i_{k+1}-i_k)N, \\
 x^{\boldsymbol{\omega},+}_i&=\tilde{z}_{-j+1}+k  & i&=2i_kN-j, \: k\in \mathbb{Z}, \: j=0,...,2(i_k-i_{k-1})N-1.
\end{align*}
Clearly $ \boldsymbol{x}^{\omega,-} < \boldsymbol{x}^{\omega} < \boldsymbol{x}^{\omega,+}$. Let
$ \mathcal{A}_{\boldsymbol{\omega}}$ be the set of all $\boldsymbol{x} \in \mathcal{X}_1$ such that $ \boldsymbol{x}^{\omega,-} \leq \boldsymbol{x} \leq \boldsymbol{x}^{\omega,+}$.

Let $\boldsymbol{h} \mapsto \boldsymbol{h}^{\tau}$ be the truncation map $\mathcal{X} \rightarrow l^2(\mathbb{Z})$ given with
\begin{equation}
h^{\tau}_i = \begin{cases} h_i & -N+1 \leq i \leq N, \\
0 & \text{otherwise}.
\end{cases}
\end{equation}
Now by the definition of $\Delta_1$, we can choose $0\leq e_0 \leq \Delta_0/2$ such that for all $\boldsymbol{h} \in \mathcal{N}(e_0)$, we have
\begin{equation}
|| \nabla E(\boldsymbol{h}) ||^2_{l^2(\mathbb{Z})} \geq \Delta_1 / 2. \label{r:lowdelta1}
\end{equation}
We define $\mathcal{B}_{\boldsymbol{\omega}}$ to be the set of all $\boldsymbol{x} \in \mathcal{X}_1$ such that
$\boldsymbol{x}^{\tau,k}:=\left[ S^{2Nj_k} \boldsymbol{x}-k - \boldsymbol{z} \right]^{\tau}$ satisfies 
\begin{equation*}
 E\left( \boldsymbol{x}^{\tau,k} \right) \leq e_0,
\end{equation*}
and furthermore that $\boldsymbol{x}^{\tau,k}$ is in the connected component of $\boldsymbol{0}$ of the set $\lbrace E(\boldsymbol{h})\leq e_0, \: \boldsymbol{h} \in l^2(\mathbb{Z}) \rbrace$, for all $k \in \mathbb{Z}$. We have that $\boldsymbol{x}^{\boldsymbol{\omega}}\in \mathcal{B}_{\boldsymbol{\omega}}$ as by definition $(\boldsymbol{x}^{\boldsymbol{\omega}})^{\tau,k}=\boldsymbol{0}$ and $E(\boldsymbol{0})=0$.
\begin{proposition} \label{p:invariant}
Assume $N$ satisfies (\ref{r:Ncond}). Then for any $\boldsymbol{\omega} \in \Omega$, $\mathcal{A}_{\omega} \cap \mathcal{B}_{\omega}$ is a $\xi$-invariant, closed set, non-empty as $\boldsymbol{x}_{\boldsymbol{\omega}} \in \mathcal{A}_{\omega} \cap \mathcal{B}_{\omega}$. 
\end{proposition}
The proof is based on the following nesting trick from \cite{Slijepcevic:16} (for the convenience of the reader, we repeat a short proof of Lemma \ref{l:nesting} below). Assume for the moment that $\xi$ is an abstract continuous semiflow on a separable metric space $\mathcal{X}$, and let $\mathcal{A}$, $\mathcal{B}$ be subsets of $\mathcal{X}$. Assume they satisfy the following for every semiorbit $q(t), t\geq t_0$ of the semiflow $\xi$:

\begin{itemize}
	\item[(B1)] $\mathcal{B}$ is $\mathcal{A}$-relatively $\xi-$invariant set. It means that if $q(t_0) \in \mathcal{B}$, and if there exists $t_1 > t_0$ such that for all $t \in [t_0,t_1]$, $q(t)\in  \mathcal{A}$, then for all $t \in [t_0,t_1]$, $q(t)\in  \mathcal{B}$.
	
	\vspace{1ex}
	\item[(B2)] There exists $\lambda >0$ such that, if $q(t_0) \in \mathcal{A} \cap \mathcal{B}$, then for all $t \in [t_0,t_0+\lambda]$,  $q(t) \in \mathcal{A}$.
\end{itemize}

\begin{lemma} \label{l:nesting}
	Assume $\mathcal{A}$, $\mathcal{B}$ are subsets of a separable metric space $\mathcal{X}$ satisfying (B1), (B2) with respect to a continuous semiflow $\xi$ on $\mathcal{X}$. Then $\mathcal{A} \cap \mathcal{B}$ is $\xi$-invariant.
\end{lemma}

\begin{proof}
	Assume the contrary and find a semi-orbit $q(t)$ of $\xi$, $t \geq t_0$, $q(t_0) \in \mathcal{A}\cap \mathcal{B}$ which violates the conclusion of the Lemma. Let 
	$$ t_2 = \sup \left\lbrace t_1\geq 0, \: q(t)\in \mathcal{A}\cap \mathcal{B}
	\text{ for all } t\in [0,t_1] \right\rbrace .$$
	Then if $t_3 = \max \lbrace t_0, t_2-\lambda/2 \rbrace$, by construction $q(t_3) \in \mathcal{A}\cap \mathcal{B}$. Now by (B2), for all $t \in [t_3,t_3 + \lambda]$, we have $q(t) \in \mathcal{A}$, and by (B1), for all $t \in [t_3,t_3 + \lambda]$ we obtain $q(t) \in \mathcal{B}$. But $t_3 + \lambda > t_2$, which is a contradiction. 
\end{proof} 
 
To complete the proof of Proposition \ref{p:invariant}, it suffices to show that $\mathcal{A}_{\boldsymbol{\omega}}$ and $\mathcal{B}_{\boldsymbol{\omega}}$ satisfy (B1), (B2).

As noted in the introduction, in the first step we use an energy argument adapted from applications to PDEs, specifically the energy-energy dissipation-energy flux balance law, where in our case energy is the truncated action.

\begin{lemma} \label{l:balance}  {\bf The action flux balance law.} Assume $k=0$, $j_0=0$, and let $\boldsymbol{h}(t)=\boldsymbol{x}(t)^{\tau,0}=[\boldsymbol{x}(t)-\boldsymbol{z}]^{\tau}$. Then
\begin{equation}
\frac{dE(\boldsymbol{h}(t))}{dt} = F(t) - || \nabla E(\boldsymbol{h}(t)) ||^2_{l^2(\mathbb{Z})}, \label{r:balance}
\end{equation}
where $F(t)$ satisfies
\begin{align}
|F(t)| & \leq \kappa_2  \max \lbrace |x_i(t)-z_i|, \: i\in \lbrace -N,-N+1,N,N+1 \rbrace \rbrace, \label{r:balance1} \\
\kappa_2 & = 8 \max \lbrace |V_1(x,y)|,|V_2(x,y)|, \: |x-y| \leq 2  \rbrace \cdot \max \lbrace |V_{12}(x,y)|,|x-y| \leq 2 \rbrace. \label{r:balance2}
\end{align}
\end{lemma}
The proof is a straightforward calculation, given in Appendix B. We will now show that the flux term $F(t)$ is by (\ref{r:condhyp}) and (\ref{r:balance1}) exponentially small $\sim \lambda^{-N}$ whenever $x(t)\in \mathcal{A}_{\boldsymbol{\omega}}$ (in the case $j_0=0$). We use this to bound $|F(t)|$ from above, and the definition of $\mathcal{B}_{\boldsymbol{\omega}}$ to bound the dissipation term $|| \nabla E(\boldsymbol{h}(t_2)) ||^2_{l^2(\mathbb{Z})}$ in (\ref{r:balance}) from below, and to establish the following:

\begin{lemma}  Assume (\ref{r:Ncond}) holds. Then the sets $\mathcal{A}_{\boldsymbol{\omega}}$ and $\mathcal{B}_{\boldsymbol{\omega}}$ satisfy (B1).
\end{lemma}

\begin{proof}
Let $k=0$, assume $j_0=0$, and set $\boldsymbol{h}(t)=\boldsymbol{x}(t)^{\tau,0}=[\boldsymbol{x}(t)-\boldsymbol{z}]^{\tau}$.
 Assume $\boldsymbol{x}(t) \in \mathcal{A}_{\boldsymbol{\omega}}$ for all $t \in [t_0,t_1]$, $t_1 > t_0$, and that
\begin{equation}
 \boldsymbol{h}(t) \text{ is in the connected component of } \boldsymbol{0} \text{ of the set } \lbrace E(\boldsymbol{g})\leq e_0, \: \boldsymbol{g} \in l^2(\mathbb{Z}) \rbrace \label{r:property}
\end{equation} 
for $t=t_0$. We now show  (\ref{r:property}) holds for all $t \in [t_1,t_2]$.  
 Assume the contrary, i.e. that (\ref{r:property}) fails for some $t \in [t_0,t_1]$. Let $t_2$ be the infimum of such $t$, and then $t_0 < t_2 < t_1$ and by assumption $E(\boldsymbol{h}(t_2))=e_0$. 
 
 We first show that 
 \begin{equation} 
 \boldsymbol{h}(t_2) \in \mathcal{N}(e_0). \label{r:Nclaim}
\end{equation}
 It suffices to show that for all $k \in \mathbb{Z}$, $|h_i(t_2)| \leq u_i = 2\kappa_1 \lambda^{-|i|}$. We use that $\boldsymbol{x}(t_2) \in \mathcal{A}_{\boldsymbol{\omega}}$, thus 
 $$
 \boldsymbol{x}^{\boldsymbol{\omega},-} \leq \boldsymbol{x}(t_2) \leq \boldsymbol{x}^{\boldsymbol{\omega},+},
 $$
 and by definition for any $-N+1 \leq i \leq N$,
 $$
 x^{\boldsymbol{\omega},-}_i \leq z_i \leq x^{\boldsymbol{\omega},+}_j,
 $$
 thus for all $-N+1 \leq i \leq N$,
 \begin{equation}
 |h_i(t_2)| \leq x^{\boldsymbol{\omega},+}_i - x^{\boldsymbol{\omega},-}_i. \label{r:esth}
 \end{equation}
 Consider $-N+1 \leq i \leq 0$. Then inserting $i=-2i_{-1}N+j$ in the definition of $x^{\boldsymbol{\omega},-}$ and $i=2i_0N-j=-j$ in the definition of $x^{\boldsymbol{\omega},+}$, and using the fact that $j \mapsto \tilde{z}_j$ is increasing and $i_{i_{-1}}\leq -1$, we get
 \begin{align*}
 x_i^{\boldsymbol{\omega},-} & \geq \tilde{z}_{-2i_{-1}N+i} -1 \geq \tilde{z}_{2N+i}-1 \geq \tilde{z}_{-i} -1 \geq \tilde{z}_{|i|-1}-1, \\
 x_i^{\boldsymbol{\omega},+} & \leq \tilde{z}_{i+1}=\tilde{z}_{-(|i|-1)},
 \end{align*}
which combined with (\ref{r:condhyp}) yields $x^{\boldsymbol{\omega},+}_i - x^{\boldsymbol{\omega},-}_i \leq 2 \kappa \lambda^{-|i|}$, as required. Analogously we show the same for $1 \leq i \leq N$.
As by definition, $h_i(t_2)=0$ for any $i$ outside $-N+1,...,N$, we have shown (\ref{r:Nclaim}).

Now we see that by (\ref{r:lowdelta1}),
 \begin{equation}
 || \nabla E(\boldsymbol{h}(t_2)) ||^2_{l^2(\mathbb{Z})} \geq \Delta_1 / 2. \label{r:lowerdiss}
\end{equation}
Now we bound the flux $|F(t_2)|$ from above. Analogously as above, we deduce that for $i \in \lbrace -N,-N+1,N,N+1 \rbrace$, we have that
$
|x_i(t)-z_i| \leq \kappa_1 \lambda^{-N},
$
thus by (\ref{r:balance1}),
\begin{equation}
 |F(t_2)| \leq \kappa_1 \kappa_2 \lambda^{-N}. \label{r:upperflux}
\end{equation}
Inserting that in (\ref{r:balance}), we get that
$$
\frac{dE(\boldsymbol{h}(t_2))}{dt}  \leq \kappa_1 \kappa_2 \lambda^{-N} - \frac{\Delta_1}{2},
$$
which is $<0$ whenever (\ref{r:Ncond}) holds, thus (\ref{r:property}) holds on some interval $[t_2,t_2+\delta]$, $\delta >0$, which is a contradiction. Thus (\ref{r:property}) holds for all $t \in [t_1,t_2]$, as claimed. We prove an analogous claim for any combination of $k,j_k$ (e.g. by reducing it to the case $k=0$, $j_0=0$ by a translation), which yields (B1).
\end{proof}

\begin{lemma} \label{l:two} Assume (\ref{r:Ncond}) holds. Then the sets $\mathcal{A}_{\boldsymbol{\omega}}$ and $\mathcal{B}_{\boldsymbol{\omega}}$ satisfy (B2).
\end{lemma}

\begin{proof}  We first prove the following claim (we will later see that it suffices): if $\boldsymbol{x}(t_0) \in \mathcal{B}_{\boldsymbol{\omega}}$ in the case $i_0=0$, and if
\begin{equation} \begin{aligned}
x_{-j}(t) & \leq \tilde{z}_{-j+1},& j &\geq 0,  \\
 x_j(t) & \geq \tilde{z}_j, & j& \geq 1, 
\end{aligned} \label{r:whatever} \end{equation}
for $t=t_0$, then there is an absolute $\lambda > 0$ such that (\ref{r:whatever}) holds for $t \in [t_0,t_0+ \lambda]$. 

Let $E_0,E_1 \subset (y_0,y_0+1)$ be the connected components of $z_0$, $z_1$ respectively, of the set $\lbrace S(x) \leq e_0 \rbrace$. By continuity of $S$, $E_0 = [a_0,b_0]$, $E_1 = [a_1,b_1]$, and by the definition of $e_0$, $\hat{\boldsymbol{z}}$, $\boldsymbol{z}$ and $S$, we have that
\begin{equation}
\tilde{z}_0 < a_0 \leq b_0 < \tilde{z}_1 \label{r:E0}
\end{equation}
 and $\tilde{z}_1 < a_1 \leq b_1$. Let $\delta = \min \lbrace \tilde{z}_1-b_0,a_1-\tilde{z}_1 \rbrace$, $\delta > 0$. It is easy to check from (\ref{r:flow}) that $|d x_k(t)/dt|$ is bounded uniformly in $k\in \mathbb{Z}$, $t$ on $\mathcal{X}_1$, thus there exists $\lambda > 0$ such that 
\begin{equation}
\begin{aligned}
&\text{If }x_0(t_0) \in E_0, &\text{ then for all }t \in [t_0,t_0+\lambda], \:\: &x_0(t)\leq \tilde{z}_1, \\
&\text{If }x_1(t_0) \in E_1, &\text{ then for all }t \in [t_0,t_0+\lambda], \: \: &x_1(t)\geq \tilde{z}_1. 
\end{aligned} \label{r:claim} \end{equation}
Finally, it is easy to check that if $\boldsymbol{x}(t_0) \in \mathcal{A}_{\boldsymbol{\omega}} \cap \mathcal{B}_{\boldsymbol{\omega}}$ in the case $i_0$, then $[\boldsymbol{x}(t_0)]^{\tau,0} \in \mathcal{N}(e_0)$, thus by the definition of $S$ we have that $x_0(t_0) \in E_0$ and $x_1(t_0) \in E_0$. From this, (\ref{r:claim}) and the order-preserving property Lemma \ref{l:order} of the semiflow for rays we deduce that (\ref{r:whatever}) holds for $t \in [t_0,t_0+\lambda]$. 

Now, it is easy to check that $\boldsymbol{x}(t_0) \in \mathcal{A}_{\boldsymbol{\omega}}$ in the case $i_0=0$ implies (\ref{r:whatever}) for $t=t_0$. We show an analogous claim for an arbitrary $k,j_k$ by substituting $\boldsymbol{x}$ with $(S^{2Nj_k}{\boldsymbol{x}}-k)$ in (\ref{r:whatever}). Finally, if (\ref{r:whatever}) holds for $t \in [t_0,t_0 + \lambda]$ for all $(S^{2Nj_k}{\boldsymbol{x}}-k)$, $k\in \mathbb{Z}$, by definition $\boldsymbol{x}(t) \in \mathcal{A}_{\boldsymbol{\omega}}$ for $t \in [t_0,t_0 + \lambda]$.
\end{proof}

\begin{remark} \label{rem:E0} Assume $\boldsymbol{x} \in \mathcal{A}_{\boldsymbol{\omega}} \cap \mathcal{B}_{\boldsymbol{\omega}}$. We have established in the proof of Lemma \ref{l:two} that if  $\omega_0 = 1$, then $x_0 \in E_0=[a_0,b_0]$, where $a_0,b_0$ satisfy (\ref{r:E0}). It is easy to check from the definition of $\mathcal{A}_{\boldsymbol{\omega}}$ that if $\omega_0 =0$, then $\hat{z}_1-1 \leq x_0 \leq \hat{z}_0$. We will need that later.
\end{remark}

\section{Proofs of the main results}

Consider the dynamical system Bernoulli shift $(\Omega,\mathcal{B}(\Omega),\sigma,\mu_{\Omega})$, where $\Omega=\lbrace 0,1 \rbrace^{\mathbb{Z}}$ as in Section \ref{s:invariant}, $\mathcal{B}(\Omega)$ is the Borel $\sigma$-algebra, $\sigma$ is the shift map, and $\mu_{\Omega}$ the product (Borel probability) measure on $(\Omega, \mathcal{B}(\Omega))$. It is well-known \cite{Katok:95} that the metric entropy $h_{\mu_{\Omega}}(\sigma)=\log 2$. We will prove the main results by constructing a $f^{2N}$-invariant measure which has $(\Omega,\mathcal{B}(\Omega),\sigma,\mu_{\Omega})$ as a factor, by combining Lemma \ref{l:push} and Propositions \ref{p:construction}, \ref{p:invariant}, in the light of Remark \ref{r:N}.

\begin{proof}[Proof of Theorem \ref{t:main1}] Fix the smallest integer $N \geq 1$ satisfying (\ref{r:Ncond}). Let $\chi : \Omega \rightarrow \mathcal{X}_1$, $\chi(\boldsymbol{\omega})=\boldsymbol{x}^{\boldsymbol{\omega}}$, and let $\hat{\chi}: \Omega \rightarrow \hat{\mathcal{X}}_1$, $\hat{\chi}=\pi \circ \chi$, $\boldsymbol{\omega} \mapsto \hat{\boldsymbol{x}}^{\boldsymbol{\omega}}$, $\hat{\mu}_{\Omega}=\hat{\chi}^*\mu_{\Omega}$. Then by construction, $(\hat{\mathcal{X}}_1, \mathcal{B}(\hat{\mathcal{X}}_1),\hat{S}^{2N}, \hat{\mu}_{\Omega})$ is conjugate to $(\Omega,\mathcal{B}(\Omega),\sigma,\mu_{\Omega})$, thus 
\begin{equation}
 h_{\hat{\mu}_{\Omega}}(\hat{S}^{2N})=\log 2.
\end{equation}
We now construct $\nu \in \mathcal{M}(\hat{\mathcal{E}},\hat{S}^{2N})$ $\mathcal{G}$-shadowing $\hat{\mu}_{\Omega}$, where $\mathcal{G}$ is a $\sigma$-subalgebra satisfying (M1)-(M4) in the sense of Remark \ref{r:N}, to be constructed below. Recall the definition of the interval $E_0 = [a_0,b_0]$ in the proof of Lemma \ref{l:two}. Let $\hat{\mathcal{D}}_{\omega,k}$, $(\omega,k) \in \mathcal{I}$, $\mathcal{I}=\lbrace 0,1 \rbrace \times \mathbb{Z}$, be the set of all $\hat{x} \in \hat{\mathcal{X}}_1$ such that
\begin{equation*}
  \hat{x}_{2Nk} \in \begin{cases} [a_0,b_0] \mod 1 & \omega = 1, \\
 [ \tilde{z}_1-1,\tilde{z}_0] \mod 1 & \omega=0 .
 \end{cases}
\end{equation*}
Let $\mathcal{G}$ be the $\sigma$-algebra generated by $\hat{\mathcal{D}}_{\omega,k}$, $(\omega,k)\in \mathcal{I}$. By (\ref{r:E0}), for all $k \in \mathbb{Z}$, $\hat{\mathcal{D}}_{0,k}$ and $\hat{\mathcal{D}}_{1,k}$ are disjoint, and by definition closed, which yields (M3). Let
\begin{align*}
\mathcal{M}_1 &=\lbrace \hat{\boldsymbol{x}}^{\boldsymbol{\omega}}, \: \boldsymbol{\omega}\in \Omega \rbrace, \\
 \hat{\mathcal{D}}_{\hat{\boldsymbol{x}}^{\boldsymbol{\omega}}} & = \cap_{k \in \mathbb{Z}} \hat{\mathcal{D}}_{\omega_k,k}, \\
 \mathcal{M}_2 & = \cup_{\boldsymbol{\omega} \in \Omega} \hat{\mathcal{D}}_{\hat{\boldsymbol{x}}^{\boldsymbol{\omega}}}.
\end{align*}
It is straightforward to verify (M1),(M2). Let $\hat{\mathcal{C}}_{\boldsymbol{\omega}}=\pi (\mathcal{A}_{\boldsymbol{\omega}} \cap \mathcal{B}_{\boldsymbol{\omega}})$. By Proposition \ref{p:invariant}, $\hat{\mathcal{C}}_{\boldsymbol{\omega}}$ is $\hat{\xi}$-invariant for any $\boldsymbol{\omega}\in \Omega$. 
By construction  we have that $\hat{\boldsymbol{x}}^{\boldsymbol{\omega}} \in \hat{\mathcal{C}}_{\boldsymbol{\omega}}$, and by Remark \ref{rem:E0} and the definition of the sets we have $\hat{\mathcal{C}}_{\boldsymbol{\omega}} \subset \hat{\mathcal{D}}_{\hat{\boldsymbol{x}}^{\boldsymbol{\omega}}}$ which yields (M4).

Now by Proposition \ref{p:construction} we establish existence of an ergodic $\nu \in \mathcal{M}(\hat{\mathcal{E}},\hat{S}^{2N})$ which $\mathcal{G}$- shadows $\hat{\mu}_{\Omega}$, thus by Lemma \ref{l:entropy}, $h_{\nu}(\hat{S}^{2N})\geq \log 2$. By Lemma \ref{l:push} and the variational principle for metric and topological entropy \cite[Theorem 4.5.3]{Katok:95}, we deduce that $h_{\text{top}}(f^{2N}) \geq \log 2$. To complete the proof, it suffices to use $h_{\text{top}}(f^{2N})=2N h_{\text{top}}(f)$ \cite[Proposition 3.1.7]{Katok:95} and (\ref{r:Ncond}).
\end{proof}

Recall now the definition of the $L_1$-Wasserstein distance of measures, required in the proof below. Let $d$ be the canonical metric on $\mathbb{S}^1 \times \mathbb{R}$, and $\mu$, $\nu$ be two Borel probability measures on $\mathbb{S}^1 \times \mathbb{R}$. Then 
$$
W_1 (\mu,\nu) = \inf_{\gamma \in \Gamma(\mu,\nu)} \int_{(\mathbb{S}^1 \times \mathbb{R}) \times (\mathbb{S}^1 \times \mathbb{R})} d(x,y) d \gamma(x,y),
$$
where $\Gamma(\mu,\nu)$ is the collection of all Borel probability measures on $(\mathbb{S}^1 \times \mathbb{R}) \times (\mathbb{S}^1 \times \mathbb{R})$ whose marginals are $\mu$, respectively $\nu$. It is well-known that $W_1$ induces weak$^*$-topology on the set of probability measures on $\mathbb{S}^1 \times \mathbb{R}$.

\begin{proof}[Proof of Theorem \ref{t:main2}]
Let $N \geq 1$ be the smallest integer satisfying (\ref{r:Ncond}), and construct $\nu_n \in \mathcal{M}(\hat{\mathcal{E}},\hat{S}^{2nN})$ as in the proof of Theorem \ref{t:main1}, with $nN$ instead of $N$, $n \geq 1$ an integer. Then by construction and Lemma \ref{l:push}, the measure
\begin{equation*}
\mu_n= \dfrac{1}{2nN}\sum_{k=-nN}^{nN-1}\hat{i}^*(S^k)^*\nu_n
\end{equation*}
is a $f$-invariant measure. Choose the following coupling
$$
 \dfrac{1}{2nN}\sum_{k=-nN}^{nN-1} \left( \hat{i}^*(S^k)^*\nu_n \right) \otimes \delta{(y_0,p_0)},
$$
and insert in the definition of $W_1$. We obtain the bound (\ref{r:wasserstein}) by using the definition $\mathcal{A}_{\boldsymbol{\omega}}$ and the bounds (\ref{r:condhyp}). We omit the details of the routine calculation.
\end{proof}

\section{Remarks and examples} \label{s:remarks}

\noindent 1. The condition (A3) could be replaced with a weaker one of convergence at infinity uniformly in $x$ of $\lim_{|y| \rightarrow \infty} V(x,x+y)= \infty$. Indeed, in construction we need values of $V(x,y)$ and its derivatives only on the set $|x-y| \leq 2$; we can always modify $V$ outside this set so that (A1-3) hold.

\vspace{1ex}

\noindent 2. The condition (N1), assumed for simplicity of the statements and the proofs, is not needed for Theorems \ref{t:main1} and \ref{t:main2} to hold. Hyperbolicity of $(y_0,p_0)$ is used in (\ref{r:condhyp}) and implicitly in the fact that $|u_j|$ used in the definition of $\mathcal{N}(e)$ is in $l^2(\mathbb{Z})$. The condition (\ref{r:condhyp}) could be replaced with a more general one of finding $N$ such that $|\tilde{z}_{-N}-y_0| \ll \Delta_1$, $|\tilde{z}_N-y_0-1| \ll \Delta_1$; the definition of the $\mathcal{N}(e)$ should then be appropriately adjusted (we omit the details). Uniqueness of the action-minimizing fixed point is also not needed: one could construct required invariant measures in a gap between two action-minimizing fixed points as long as (N2) holds.

\vspace{1ex}

\noindent 3. In the case of the Standard map, it is elementary to check that if $k \neq 0$, then $(0,0)$ is the unique action-minimizing fixed point. The fact that (N2) holds for small $|k|>0$ follows from the result by Lazutkin and Gelfreich that the angle of splitting of separatrices is $>0$ whenever $k\neq 0$ \cite{Gelfreich:99}, Remarks \ref{r:41}, \ref{r:43} and Corollary \ref{c:25}. For larger $k$, it can be established e.g. by applying the criteria of Angenent \cite{Angenent:92} for non-existence of invariant circles.

\vspace{1ex}

\noindent 4. Our definition of the Peierls barrier differs from the one chosen by Aubry and Mather \cite{Aubry:88,Gole:01,Sorrentino:10} in the case of rational rotation numbers, as they consider difference in action between the minimax periodic and the minimizing periodic orbit. It is, however, the same in the case of irrational rotation numbers, to which our approach generalizes without much difficulty (see 9.7 below).

\vspace{1ex}

\noindent 5. In \cite{MacKay:88}, the flux through a gap of a Mather set was considered for irrational rotation numbers only, in which case the Mather set is a Cantori whenever it has a gap. We use the fact that the definition of $\Delta_0$, as used here, extends naturally also to rational rotation numbers, in this paper $\rho=0$.

\vspace{1ex}

\noindent 6. One can relatively easily show by elementary methods that in the case of the Standard map, for $\boldsymbol{h} \in \cup_{0 \leq e \leq \Delta_0} \mathcal{N}(e)$, we have $E(\boldsymbol{h})=-k \cos (2\pi (z_0+h_0))/2\pi + O(1)$, $||\nabla E(\boldsymbol{h})||_{l^2(\mathbb{Z})}= k |\sin (2\pi (z_0+h_0))| + O(1)$, which yields $\Delta_1= k^2 + O(k)$.

\vspace{1ex}

\noindent 7. The results of the paper can be extended to irrational rotation numbers. Analogously to the approach in Section \ref{s:positivity}, one can show that if the Mather set with the rotation number $\rho \in \mathbb{R} \setminus \mathbb{Q}$ is a Cantori, then $\Delta_1(\rho) >0$. We conjecture that the topological entropy can then be bounded from below by an expression depending on $\Delta_1(\rho)$, on number-theoretical properties of $\rho$ (related to recurrence properties and return times of the dynamics restricted to the Cantori), and on the size of the paratingent cones of the Mather set as a generalized notion of hyperbolicity of Mather sets in line with the results of M.-C. Arnaud \cite{Arnaud:11}.

\section*{Appendix A: proof of Proposition \ref{p:construction}}

Assume all the notation from Section \ref{s:construction}. The proof of Proposition \ref{p:construction} is in two steps: we first construct a measure $\nu \in \mathcal{M}(\hat{\mathcal{E}})$ by applying the LaSalle principle for gradient-like flows, and then show that it indeed $\mathcal{G}$-shadows $\mu$.

Let $\hat{\xi}^*$ be the pushed semi-flow on $\mathcal{M}(\hat{\mathcal{X}}_K)$. It is continuous and well-defined (i.e. $\hat{\xi}^{*t} \in \mathcal{M}(\hat{\mathcal{X}}_K)$ for $t \geq 0$), as $\hat{\xi}$ and $\hat{S}$ commute and as $\hat{\mathcal{X}}_K$ is $\hat{\xi}$-invariant.

\begin{lemma}
	(i) The semi-flow $\hat{\xi}^*$ is gradient-like, with the Lyapunov function
	$$ L(\mu)=\int_{\hat{\mathcal{X}}_K}V(x_0,x_1)d\mu(\boldsymbol{x}),$$
	$L$ continuous and bounded.
	
	(ii) The set of equilibria of $\hat{\xi}^*$ is $\mathcal{M}(\hat{\mathcal{E}}\cap \hat{\mathcal{X}}_K)$.
	
	(iii) For each $\mu$ in $\mathcal{M}(\hat{\mathcal{X}}_K)$, the $\omega$-limit set $\omega(\mu)$ with respect to $\hat{\xi}^*$ is a non-empty subset of $\mathcal{M}(\hat{\mathcal{E}}\cap \hat{\mathcal{X}}_K)$.
\end{lemma}

\begin{proof}
	First note that $\mathcal{M}(\hat{\mathcal{X}}_K)$ is a compact metrizable space (with metric e.g. the Wasserstein $L_1$-distance), thus standard results for continuous semiflows apply.
	
	As because of (A1), $V(x_0+n,x_1+n)=V(x_0,x_1)$ for an integer $n$, $L$ is well-defined, i.e. the integrand is independent of a representative of $\boldsymbol{x} \in \hat{\mathcal{X}}$ in $\mathcal{X}$ (this also holds for all the integrands in the calculation below). Continuity of $L$ follows from the definition of weak$^*$-topology, and boundedness from boundedness of $\boldsymbol{x} \mapsto V(x_0,x_1)$ on compact $\hat{\mathcal{X}}_K$.
	
	Differentiating with respect to $t$, we get
	\begin{align*}
	\frac{d L(\mu(t))}{dt} & = \int_{\hat{\mathcal{X}}_K}\left( V_1(x_0,x_1)\dot{x}_0 + V_2(x_0,x_1)\dot{x}_1 \right) d\mu(t)(\boldsymbol{x}) \\
	 & = - \int_{\hat{\mathcal{X}}_K}V_1(x_0,x_1)(V_2(x_{-1},x_0)+V_1(x_0,x_1))d\mu(t)(\boldsymbol{x}) \\
	 & \hspace{20ex} - \int_{\hat{\mathcal{X}}_K}
	V_2(x_0,x_1)(V_2(x_0,x_1)+V_1(x_1,x_2))  d\mu(t)(\boldsymbol{x}) \\
	& = - \int_{\hat{\mathcal{X}}_K}V_1(x_0,x_1)(V_2(x_{-1},x_0)+V_1(x_0,x_1))d\mu(t)(\boldsymbol{x}) \\
	& \hspace{20ex} - \int_{\hat{\mathcal{X}}_K}
	V_2(x_{-1},x_0)(V_2(x_{-1},x_0)+V_1(x_0,x_1)) d\mu(t)(\boldsymbol{x}) \\
	& = - \int_{\hat{\mathcal{X}}_K}(V_2(x_{-1},x_0)+V_1(x_0,x_1))^2 d\mu(t)(\boldsymbol{x}) = - \int_{\hat{\mathcal{X}}_K}\dot{x}_0(t)^2 d\mu(t)(\boldsymbol{x}),
	\end{align*}
where in the first row we swapped the derivative and the integral (possible because of (A1-3) and compactness of $\hat{\mathcal{X}}_K$); and in the third row we applied $\hat{S}$-invariance of $\mu(t)$. We see that $dL(\mu(t))/dt \leq 0$. Applying again $\hat{S}$-invariance to the last row, we deduce that $dL(\mu(t))/dt= 0$ if and only if $\mu(t)(\hat{\mathcal{E}})=0$, which is equivalent to $\mu(t)$ being an equilibrium of $\hat{\xi}^*$, which gives (i) and (ii). (iii) now follows from the LaSalle principle for continuous gradient-like semiflows on compact metric spaces.
\end{proof}

\begin{lemma} Assume (M1)-(M4). Then any $\nu \in \omega(\mu)$ $\mathcal{G}$-shadows $\mu$.
\end{lemma}

\begin{proof}
	It suffices show that $\hat{\theta}: \mathcal{M}_2 \rightarrow \mathcal{M}_1$ defined by (M2) is the factor map, and that for all $\mathcal{D} \in \mathcal{G}$, $\mu(\mathcal{D})=\nu(\mathcal{D})$. By (M2), $\hat{\theta}$ is measurable, and by (M1), $\hat{\theta} \circ \hat{S} = \hat{S} \circ \hat{\theta}$ on $\mathcal{M}_2$. Denote by $\mu(t)=\hat{\xi}^{*t}\mu$. By (M4), $\hat{\xi}^t(\mathcal{M}_1 \cap \mathcal{D})\subset  \mathcal{D}$, thus by definition 
	\begin{equation}
	\mu(t)(\mathcal{D}) \geq  \mu(\mathcal{M}_1 \cap \mathcal{D}). \label{r:geq}
	\end{equation}
	
	Take any $\mathcal{D}_i$, $i \in \mathcal{I}$ as in (M3). As $\mathcal{D}_i$ is closed, by the characterization of the weak$^*$ convergence, and as $\mu(t_n) \rightarrow \nu$ for some sequence $t_n$, we have that 
	\begin{equation}
	\nu(\mathcal{D}_i) \geq \limsup_{n \rightarrow \infty}\mu(t_n)(\mathcal{D}_i) \geq \mu(\mathcal{M}_1 \cap \mathcal{D}_i). \label{r:geq2}
	\end{equation}
		
	We now show that $\mu(\mathcal{M}_1 \cap \mathcal{D}_i) = \mu(\mathcal{D}_i)$ for any $i \in \mathcal{I}$. Indeed, by (M3) and $\mu(\mathcal{M}_1)=1$ given in (M1),
	$$
	1 = \mu \left( \cup_{j =1 }^{\infty} \mathcal{D}_{i_j} \right) = \sum_{j =1 }^{\infty} \mu(\mathcal{D}_{i_j}) \geq \sum_{j =1 }^{\infty} \mu(\mathcal{D}_{i_j} \cap \mathcal{M}_1) =\mu \left( \left( \cup_{j =1 }^{\infty} \mathcal{D}_{i_j}\right)\cap \mathcal{M}_1 \right) =1,
	$$
	thus equality must hold everywhere. Analogously from (\ref{r:geq2}) we obtain 
	$$
	1 \geq \nu \left( \cup_{j =1 }^{\infty} \mathcal{D}_{i_j} \right) = \sum_{j =1 }^{\infty} \nu(\mathcal{D}_{i_j}) \geq \sum_{j =1 }^{\infty} \mu(\mathcal{D}_{i_j} \cap \mathcal{M}_1) =\mu \left( \left( \cup_{j =1 }^{\infty} \mathcal{D}_{i_j}\right)\cap \mathcal{M}_1 \right) =1,
	$$
	thus $\nu(\mathcal{D}_i)=\mu(\mathcal{D}_i)=\mu(\mathcal{D}_i \cap \mathcal{M}_1)$ for all $i \in \mathcal{I}$. As by (M1) and (M2), $\hat{\theta}^{-1}(\mathcal{D}_i \cap \mathcal{M}_1)= \mathcal{D}_i$, as $(\mathcal{D}_i, i \in \mathcal{I})$ generates $\mathcal{G}$, and as $(\mathcal{D}_i \cap \mathcal{M}_1, i \in \mathcal{I})$ generates the Borel $\sigma$-algebra on $\mathcal{M}_1$, this completes the proof that $\hat{\theta}$ is a factor map, and of the shadowing property $\mu(\mathcal{D})=\nu(\mathcal{D})$.	
\end{proof}

\section*{Appendix B: proof of Lemma \ref{l:balance}}

We use the notation from the statement of Lemma \ref{r:balance}, and write $\boldsymbol{h}$, $\boldsymbol{x}$ instead of $\boldsymbol{h}(t_2)$, $\boldsymbol{x}(t_2)$. The derivative is always evaluated at $t=t_2$. 

As $\boldsymbol{z} \in \mathcal{E}$, by definition $|| \nabla E(\boldsymbol{h}) ||^2_{l^2(\mathbb{Z})}$ is the finite sum
\begin{align*}
|| \nabla E(\boldsymbol{h}) ||^2_{l^2(\mathbb{Z})}  =& (V_2(z_{-N-1},z_{-N})+V_1(z_{-N},x_{-N+1}))^2  + (V_2(z_{-N},x_{-N+1})+V_1(x_{-N+1},x_{-N+2}))^2 \\
& + \sum_{j =-N+2}^{N-1} (V_2(x_{j-1},x_j)+V_1(x_j,x_{j+1}))^2 \\
& + (V_2(x_{N-1},x_N)+V_1(x_N,z_{N+1}))^2 + (V_2((x_N,z_{N+1}))+V_1(z_{N+1},z_{N+2}))^2.
\end{align*}
Denote the first two and the last two summands by $F_1,F_2,F_3,F_4$. Differentiating carefully we see that
\begin{align*}
 \frac{dE(\boldsymbol{h})(t)}{dt}  = &-V_2(z_{-N},x_{-N+1})(V_2(x_{-N},x_{-N+1})+V_1(x_{-N+1},x_{-N+2})) \\
 & - V_1(x_{-N+1},x_{-N+2})(V_2(x_{-N},x_{-N+1})+V_1(x_{-N+1},x_{-N+2})) \\
 & - \sum_{j =-N+2}^{N-1} (V_2(x_{j-1},x_j)+V_1(x_j,x_{j+1}))^2 \\
  & -  V_2(x_{N-1},x_N)(V_2(x_{N-1},x_N)+V_1(x_N,x_{N+1}) \\
  &- V_1(x_N,z_{N+1})(V_2(x_{N-1},x_N)+V_1(x_N,x_{N+1}).
\end{align*}
Denote the first two and the last two summands by $F_5,F_6,F_7,F_8$. We now see that (\ref{r:balance}) holds with $F(t_2)=F_1+...+F_8$. Now as $\boldsymbol{z} \in \mathcal{E}$, we have that $V_2(z_{-N-1},z_{-N})=-V_1(z_{-N},z_{-N+1})$, thus
\begin{align*}
|F_1 |  & = (V_1(z_{-N},x_{-N+1})-V_1(z_{-N},z_{-N+1}))^2 \\& \leq  |V_{12}(z_{-N},c_1)||x_{-N+1}-z_{-N+1}| \cdot (|V_1(z_{-N},x_{-N+1})|+|V_1(z_{-N},z_{N+1})|) 
\\ &  \leq \frac{\kappa_2}{4}|x_{-N+1}-z_{-N+1}|,
\end{align*}
where $c_1 \in [z_{-N+1},x_{-N+1}] $ and $\kappa_2$ as in (\ref{r:balance2}).
Now by an analogous argument in the second row below, we get
\begin{align*}
|F_2+F_5+F_6| & = |V_2(z_{-N},x_{-N+1})-V_2(x_{-N},x_{-N+1}|\cdot |V_2(z_{-N},x_{-N+1})+V_1(x_{-N+1},x_{-N+2})|
\\ &  \leq \frac{\kappa_2}{4}|x_{-N}-z_{-N}|.
\end{align*}
We similarly deduce analogous bounds on $|F_4|$, $|F_3+F_7+F_8|$, and by the triangle inequality obtain (\ref{r:balance1}).

{\small
{\em Authors' addresses}:
{\em Sini\v{s}a Slijep\v{c}evi\'{c}}, Department of Mathematics, Bijeni\v{c}ka 30, University of Zagreb, Croatia
 e-mail: \texttt{slijepce@\allowbreak math.hr}.

}

\end{document}